\documentclass{amsart}
\usepackage{amssymb}
\usepackage{amsfonts}

\newtheorem{theorem}{Theorem}
\theoremstyle{plain}

\newtheorem{corollary}{Corollary}

\newtheorem{definition}{Definition}

\newtheorem{lemma}{Lemma}

\newtheorem{proposition}{Proposition}
\newtheorem{remark}{Remark}

\numberwithin{equation}{section}
\input{tcilatex}

\begin{document}
\title[inequalities for $s-$convex functions]{New inequalities of
Ostrowski's type for $s-$convex functions in the second sense with
applications}
\author{Erhan SET$^{\star \clubsuit }$}
\address{$^{\clubsuit }$Atat\"{u}rk University, K.K. Education Faculty,
Department of Mathematics, 25240, Campus, Erzurum, Turkey}
\email{erhanset@yahoo.com}
\thanks{$^{\star }$corresponding author}
\author{M. Emin \"{O}zdemir$^{\blacksquare }$}
\address{$^{\blacksquare }$Graduate School of Natural and Applied Sciences, A%
\u{g}r\i\ \.{I}brahim \c{C}e\c{c}en University, A\u{g}r\i , Turkey}
\email{emos@atauni.edu.tr}
\author{Mehmet Zeki Sar\i kaya$^{\blacklozenge }$}
\address{$^{\blacklozenge }$Department of Mathematics,Faculty of Science and
Arts, D\"{u}zce University, D\"{u}zce, Turkey}
\email{sarikayamz@gmail.com}
\date{}

\begin{abstract}
In this paper, we establish some new inequalities of Ostrowski's type for
functions whose derivatives in absolute value are the class of s-convex.
Some applications for special means of real numbers are also provided.
Finally, some error estimates for the midpoint formula are obtained.
\end{abstract}

\maketitle

\section{Introduction}

The following result is known in the literature as Ostrowski's inequality 
\cite{AO}

\begin{theorem}
Let $f:I\subset \left[ 0,\infty \right] \rightarrow 
\mathbb{R}
$ be a differentiable mapping on $I^{\circ }$, the interior of the interval $%
I$ , such that $f^{\prime }\in L\left[ a,b\right] $ where $a$ $,$ $b\in I$
with $a<b$ . If $\left\vert f^{\prime }\left( x\right) \right\vert \leq M$,
then the following inequality holds:
\end{theorem}

\begin{equation}
\left\vert f(x)-\frac{1}{b-a}\int_{a}^{b}f(u)du\right\vert \leq M\left(
b-a\right) \left[ \frac{1}{4}+\frac{\left( x-\frac{a+b}{2}\right) ^{2}}{%
\left( b-a\right) ^{2}}\right]  \label{e.1.1}
\end{equation}

Recently, Ostrowski's inequality has been the subject of intensive research.
In particular, many generalizations , improvements , and applications for
the Ostrowski's inequality can be found in the literature (\cite{AD}-\cite%
{BCDPS},\cite{DS}-\cite{DW1},\cite{AO} and \cite{sarikaya}) and the
references therein.

In \cite{AD}, Alomari and Darus obtained inequalities for differentiable
convex mappings which are connected with Ostrowski's inequality, and they
used the following lemma to prove them. We have corrected by writting $%
\left( a-b\right) $ instead of $\left( b-a\right) $ in the right side of
this lemma.

\begin{lemma}
\label{l.1.1} Let $f:I\subset \mathbb{R}\rightarrow \mathbb{R}$ be a
differentiable mapping on $I^{\circ }\ $where $a,b\in I$ with $a<b.$ If $%
f^{\prime }\in L\left[ a,b\right] ,$ then the following equality holds:%
\begin{equation}
f(x)-\frac{1}{b-a}\int_{a}^{b}f(u)du=\left( a-b\right)
\int_{0}^{1}p(t)f^{\prime }(ta+(1-t)b)dt  \label{e.1.2}
\end{equation}%
for each $t\in \left[ 0,1\right] ,$ where%
\begin{equation*}
p\left( t\right) =\QATOPD\{ . {t,\text{ \ \ \ \ \ \ \ \ \ \ \ \ }t\in \left[
0,\frac{b-x}{b-a}\right] }{t-1,\text{ \ \ \ \ \ \ }t\in \left( \frac{b-x}{b-a%
},1\right] },
\end{equation*}%
for all $x\in \left[ a,b\right] .$
\end{lemma}

\begin{definition}
\cite{breckner} A function $f:[0,\infty )\mathbb{\rightarrow R}$ is said to
be $s-$convex in the second sense if 
\begin{equation*}
f(\alpha x+(1-\alpha )y)\leq \alpha ^{s}f(x)+(1-\alpha )^{s}f(y)
\end{equation*}%
for all $x,y\in \lbrack 0,\infty )$, $\alpha \in \lbrack 0,1]$ and for some
fixed $s\in (0,1].$ This class of $s$-convex functions is usually denoted by 
$K_{s}^{2}.$
\end{definition}

An $s-$convex function was introduced in Breckner's paper \cite{breckner}
and a number of properties and connections with $s-$convexity in the first
sense are discussed in paper \cite{hudzik}. Of course, $s-$convexity means
just convexity when $s=1.$

In \cite{dragomir1}, Dragomir and Fitzpatrick proved a variant of Hadamard's
inequality which holds for $s-$convex functions in the second sense:

\begin{theorem}
Suppose that $f:[0,\infty )\rightarrow \lbrack 0,\infty )$ is an $s$-convex
function in the second sense, where $s\in (0,1),$ and let $a,b\in \lbrack
0,\infty ),$ $a<b.$ If $f\in L^{1}(\left[ a,b\right] ),$ then the following
inequalities hold:%
\begin{equation}
2^{s-1}f(\frac{a+b}{2})\leq \frac{1}{b-a}\int\limits_{a}^{b}f(x)dx\leq \frac{%
f(a)+f(b)}{s+1}.  \label{e.1.3}
\end{equation}
\end{theorem}

The constant $k=\frac{1}{s+1}$ is the best possible in the second inequality
in (\ref{e.1.3}).

In \cite{ADDC}, Alomari et al. proved the following inequality of Ostrowski
type for functions whose derivative in absolute value are $s-$convex in the
second sense.

\begin{theorem}
Let $f:I\subset \lbrack 0,\infty )\rightarrow \mathbb{R}$ be a
differentiable mapping on $I^{\circ }$ such that $f^{\prime }\in L\left[ a,b%
\right] $, where $a,b\in I$ with $a<b.$ If $\left\vert f^{\prime
}\right\vert ^{q}$ is $s-$convex in the second sense on $\left[ a,b\right] $
for some fixed $s\in \left( 0,1\right] $, $p,q>1$, $\frac{1}{p}+\frac{1}{q}%
=1 $ and $\left\vert f^{\prime }(x)\right\vert \leq M$, $x\in \left[ a,b%
\right] $, then the following inequality holds:%
\begin{equation}
\left\vert f(x)-\frac{1}{b-a}\int_{a}^{b}f(u)du\right\vert \leq \frac{M}{%
\left( 1+p\right) ^{\frac{1}{p}}}\left( \frac{2}{s+1}\right) ^{\frac{1}{q}%
}\left\{ \frac{\left( x-a\right) ^{2}+\left( b-x\right) ^{2}}{\left(
b-a\right) }\right\}  \label{EE}
\end{equation}%
for each $x\in \left[ a,b\right] $.
\end{theorem}

In \cite{USK}, some inequalities of Hermite-Hadamard's type for
differentiable convex mappings were presented as follows:

\begin{theorem}
Let $f:I\subset \mathbb{R}\rightarrow \mathbb{R}$ be a differentiable
mapping on $I^{\circ },\ $where $a,b\in I$ with $a<b.$ If $\left\vert
f^{\prime }\right\vert $ is convex on $[a,b],$ then the following inequality
holds,%
\begin{equation}
\left\vert \frac{1}{b-a}\int_{a}^{b}f(x)dx-f\left( \frac{a+b}{2}\right)
\right\vert \leq \dfrac{\left( b-a\right) }{4}\left[ \frac{\left\vert
f^{\prime }(a)\right\vert +\left\vert f^{\prime }(b)\right\vert }{2}\right] .
\label{e.1.4}
\end{equation}
\end{theorem}

\begin{theorem}
Let $f:I\subset \mathbb{R}\rightarrow \mathbb{R}$ be a differentiable
mapping on $I^{\circ },\ $where $a,b\in I^{\circ }$ with $a<b,$ and let $%
p>1. $ If the mapping $\left\vert f^{\prime }\right\vert ^{p/(p-1)}$ is
convex on $[a,b],$ then we have%
\begin{eqnarray}
&&\left\vert \frac{1}{b-a}\int_{a}^{b}f(x)dx-f\left( \frac{a+b}{2}\right)
\right\vert  \label{e.1.5} \\
&\leq &\dfrac{\left( b-a\right) }{16}\left( \frac{4}{p+1}\right) ^{\frac{1}{p%
}}\left[ \left( \left\vert f^{\prime }(a)\right\vert ^{p/(p-1)}+3\left\vert
f^{\prime }(b)\right\vert ^{p/(p-1)}\right) ^{(p-1)/p}\right.  \notag \\
&&+\left. \left( 3\left\vert f^{\prime }(a)\right\vert ^{p/(p-1)}+\left\vert
f^{\prime }(b)\right\vert ^{p/(p-1)}\right) ^{(p-1)/p}\right] .  \notag
\end{eqnarray}
\end{theorem}

\begin{theorem}
Let $f:I\subset \mathbb{R}\rightarrow \mathbb{R}$ be a differentiable
mapping on $I^{\circ },\ $where $a,b\in I^{\circ }$ with $a<b,$ and let $%
p>1. $ If the mapping $\left\vert f^{\prime }\right\vert ^{p/(p-1)}$ is
convex on $[a,b],$ then we have%
\begin{equation}
\left\vert \frac{1}{b-a}\int_{a}^{b}f(x)dx-f\left( \frac{a+b}{2}\right)
\right\vert \leq \dfrac{\left( b-a\right) }{4}\left( \frac{4}{p+1}\right) ^{%
\frac{1}{p}}\left( \left\vert f^{\prime }(a)\right\vert +\left\vert
f^{\prime }(b)\right\vert \right) .  \label{e.1.6}
\end{equation}
\end{theorem}

The main purpose of this paper is to establish new Ostrowski's type
inequalities for the class of functions whose derivatives in absolute value
at certain powers are $s-$convex in the second sense. Also, using these
results we note some consequent applications to special means and to
estimates of the error term in the midpoint formula.

\section{Main Results}

The next theorem gives a new result of the Ostrowski's inequality for $s$%
-convex functions:

\begin{theorem}
\label{t.2.0} Let $f:I\subset \lbrack 0,\infty )\rightarrow \mathbb{R}$ be a
differentiable mapping on $I^{\circ }$ such that $f^{\prime }\in L\left[ a,b%
\right] $, where $a,b\in I$ with $a<b.$ If $\left\vert f^{\prime
}\right\vert $ is $s-$convex on $\left[ a,b\right] $, for some fixed $s\in
\left( 0,1\right] $, then the following inequality holds:%
\begin{eqnarray*}
&&\left\vert f(x)-\frac{1}{b-a}\int_{a}^{b}f(u)du\right\vert \\
&\leq &\frac{b-a}{\left( s+1\right) \left( s+2\right) } \\
&&\times \left\{ \left[ 2\left( s+1\right) \left( \frac{b-x}{b-a}\right)
^{s+2}-\left( s+2\right) \left( \frac{b-x}{b-a}\right) ^{s+1}+1\right]
\left\vert f^{\prime }(a)\right\vert \right. \\
&&\left. +\left[ 2\left( s+1\right) \left( \frac{x-a}{b-a}\right)
^{s+2}-\left( s+2\right) \left( \frac{x-a}{b-a}\right) ^{s+1}+1\right]
\left\vert f^{\prime }(b)\right\vert \right\}
\end{eqnarray*}%
for each $x\in \left[ a,b\right] .$
\end{theorem}

\begin{proof}
By Lemma \ref{l.1.1} and since $\left\vert f^{\prime }\right\vert $ is $s-$%
convex on $\left[ a,b\right] $, then we have%
\begin{eqnarray*}
&&\left\vert f(x)-\frac{1}{b-a}\int_{a}^{b}f(u)du\right\vert \\
&\leq &\left( b-a\right) \int_{0}^{\frac{b-x}{b-a}}t\left\vert f^{\prime
}\left( ta+\left( 1-t\right) b\right) \right\vert dt \\
&&+\left( b-a\right) \int_{\frac{b-x}{b-a}}^{1}\left\vert t-1\right\vert
\left\vert f^{\prime }\left( ta+\left( 1-t\right) b\right) \right\vert dt \\
&\leq &\left( b-a\right) \int_{0}^{\frac{b-x}{b-a}}t\left( t^{s}\left\vert
f^{\prime }(a)\right\vert +\left( 1-t\right) ^{s}\left\vert f^{\prime
}(b)\right\vert \right) dt \\
&&+\left( b-a\right) \int_{\frac{b-x}{b-a}}^{1}\left( 1-t\right) \left(
t^{s}\left\vert f^{\prime }(a)\right\vert +\left( 1-t\right) ^{s}\left\vert
f^{\prime }(b)\right\vert \right) dt
\end{eqnarray*}%
\begin{eqnarray*}
&=&\left( b-a\right) \left\{ \left\vert f^{\prime }(a)\right\vert \int_{0}^{%
\frac{b-x}{b-a}}t^{s+1}dt+\left\vert f^{\prime }(b)\right\vert \int_{0}^{%
\frac{b-x}{b-a}}t\left( 1-t\right) ^{s}dt\right. \\
&&\left. +\left\vert f^{\prime }(a)\right\vert \int_{\frac{b-x}{b-a}%
}^{1}\left( t^{s}-t^{s+1}\right) dt+\left\vert f^{\prime }(b)\right\vert
\int_{\frac{b-x}{b-a}}^{1}\left( 1-t\right) ^{s+1}dt\right\} \\
&=&\frac{b-a}{\left( s+1\right) \left( s+2\right) }\left\{ \left[ 2\left(
s+1\right) \left( \frac{b-x}{b-a}\right) ^{s+2}-\left( s+2\right) \left( 
\frac{b-x}{b-a}\right) ^{s+1}+1\right] \left\vert f^{\prime }(a)\right\vert
\right. \\
&&\left. +\left[ s\left( \frac{x-a}{b-a}\right) ^{s+2}-\left( s+2\right) 
\frac{b-x}{b-a}\left( \frac{x-a}{b-a}\right) ^{s+1}+1\right] \left\vert
f^{\prime }(b)\right\vert \right\}
\end{eqnarray*}%
where we use the facts that%
\begin{equation*}
\int_{0}^{\frac{b-x}{b-a}}t^{s+1}dt=\frac{1}{s+2}\left( \frac{b-x}{b-a}%
\right) ^{s+2}
\end{equation*}%
\begin{equation*}
\int_{0}^{\frac{b-x}{b-a}}t\left( 1-t\right) ^{s}dt=\frac{1}{s+2}\left( 
\frac{x-a}{b-a}\right) ^{s+2}-\frac{1}{\left( s+1\right) }\left( \frac{x-a}{%
b-a}\right) ^{s+1}+\frac{1}{\left( s+1\right) \left( s+2\right) }
\end{equation*}%
\begin{equation*}
\int_{\frac{b-x}{b-a}}^{1}\left( t^{s}-t^{s+1}\right) dt=\frac{1}{\left(
s+1\right) \left( s+2\right) }-\frac{1}{s+1}\left( \frac{b-x}{b-a}\right)
^{s+1}+\frac{1}{s+2}\left( \frac{b-x}{b-a}\right) ^{s+2}
\end{equation*}%
\begin{equation*}
\int_{\frac{b-x}{b-a}}^{1}\left( 1-t\right) ^{s+1}dt=\frac{1}{s+2}\left( 
\frac{x-a}{b-a}\right) ^{s+2}
\end{equation*}%
which completes the proof.
\end{proof}

\begin{corollary}
\label{cor1} In Theorem \ref{t.2.0}, if we choose $x=\frac{a+b}{2},$ then we
have the following midpoint inequality:%
\begin{eqnarray}
&&\left\vert f(\frac{a+b}{2})-\frac{1}{b-a}\int_{a}^{b}f(u)du\right\vert
\label{e.2.0} \\
&\leq &\frac{b-a}{\left( s+1\right) \left( s+2\right) }\left( 1-\frac{1}{%
2^{s+1}}\right) \left[ \left\vert f^{\prime }(a)\right\vert +\left\vert
f^{\prime }(b)\right\vert \right] .  \notag
\end{eqnarray}
\end{corollary}

\begin{remark}
\label{R.ek1} In Corollary \ref{cor1}, if $s=1$, then we have%
\begin{equation*}
\left\vert f(\frac{a+b}{2})-\frac{1}{b-a}\int_{a}^{b}f(u)du\right\vert \leq 
\dfrac{\left( b-a\right) }{4}\left[ \frac{\left\vert f^{\prime
}(a)\right\vert +\left\vert f^{\prime }(b)\right\vert }{2}\right]
\end{equation*}%
which is (\ref{e.1.4}).
\end{remark}

\begin{theorem}
\label{teo1} Let $f:I\subset \lbrack 0,\infty )\rightarrow \mathbb{R}$ be a
differentiable mapping on $I^{\circ }$ such that $f^{\prime }\in L\left[ a,b%
\right] $, where $a,b\in I$ with $a<b.$ If $\left\vert f^{\prime
}\right\vert ^{q}$ is $s-$convex on $\left[ a,b\right] $, for some fixed $%
s\in \left( 0,1\right] $ and $p>1$, then the following inequality holds:%
\begin{eqnarray*}
&&\left\vert f(x)-\frac{1}{b-a}\int_{a}^{b}f(u)du\right\vert \\
&\leq &\left( b-a\right) \frac{1}{\left( p+1\right) ^{\frac{1}{p}}}\frac{1}{%
\left( s+1\right) ^{\frac{1}{q}}} \\
&&\times \left\{ \left( \frac{b-x}{b-a}\right) ^{1+\frac{1}{p}}\left( \left( 
\frac{b-x}{b-a}\right) ^{s+1}\left\vert f^{\prime }(a)\right\vert ^{q}+\left[
1-\left( \frac{x-a}{b-a}\right) ^{s+1}\right] \left\vert f^{\prime
}(b)\right\vert ^{q}\right) ^{\frac{1}{q}}\right. \\
&&+\left. \left( \frac{x-a}{b-a}\right) ^{1+\frac{1}{p}}\left( \left[
1-\left( \frac{b-x}{b-a}\right) ^{s+1}\right] \left\vert f^{\prime
}(a)\right\vert ^{q}+\left( \frac{x-a}{b-a}\right) ^{s+1}\left\vert
f^{\prime }(b)\right\vert ^{q}\right) ^{\frac{1}{q}}\right\}
\end{eqnarray*}%
for each $x\in \left[ a,b\right] $, where $\frac{1}{p}+\frac{1}{q}=1.$
\end{theorem}

\begin{proof}
Suppose that $p>1.$ From Lemma \ref{l.1.1} and using the H\"{o}lder
inequality, we have 
\begin{eqnarray}
&&\left\vert f(x)-\frac{1}{b-a}\int_{a}^{b}f(u)du\right\vert  \label{E0} \\
&\leq &\left( b-a\right) \int_{0}^{\frac{b-x}{b-a}}t\left\vert f^{\prime
}\left( ta+\left( 1-t\right) b\right) \right\vert dt  \notag \\
&&+\left( b-a\right) \int_{\frac{b-x}{b-a}}^{1}\left\vert t-1\right\vert
\left\vert f^{\prime }\left( ta+\left( 1-t\right) b\right) \right\vert dt 
\notag \\
&\leq &\left( b-a\right) \left( \int_{0}^{\frac{b-x}{b-a}}t^{p}dt\right) ^{%
\frac{1}{p}}\left( \int_{0}^{\frac{b-x}{b-a}}\left\vert f^{\prime }\left(
ta+\left( 1-t\right) b\right) \right\vert ^{q}dt\right) ^{\frac{1}{q}} 
\notag \\
&&+\left( b-a\right) \left( \int_{\frac{b-x}{b-a}}^{1}\left( 1-t\right)
^{p}dt\right) ^{\frac{1}{p}}\left( \int_{\frac{b-x}{b-a}}^{1}\left\vert
f^{\prime }\left( ta+\left( 1-t\right) b\right) \right\vert ^{q}dt\right) ^{%
\frac{1}{q}}.  \notag
\end{eqnarray}%
Using the $s-$convexity of $\left\vert f^{\prime }\right\vert ^{q}$, we
obtain%
\begin{eqnarray}
&&\int_{0}^{\frac{b-x}{b-a}}\left\vert f^{\prime }\left( ta+\left(
1-t\right) b\right) \right\vert ^{q}dt  \label{E1} \\
&\leq &\int_{0}^{\frac{b-x}{b-a}}\left[ t^{s}\left\vert f^{\prime
}(a)\right\vert ^{q}+\left( 1-t\right) ^{s}\left\vert f^{\prime
}(b)\right\vert ^{q}\right] dt  \notag \\
&=&\frac{1}{s+1}\left\{ \left( \frac{b-x}{b-a}\right) ^{s+1}\left\vert
f^{\prime }(a)\right\vert ^{q}+\left[ 1-\left( \frac{x-a}{b-a}\right) ^{s+1}%
\right] \left\vert f^{\prime }(b)\right\vert ^{q}\right\}  \notag
\end{eqnarray}%
and%
\begin{eqnarray}
&&\int_{\frac{b-x}{b-a}}^{1}\left\vert f^{\prime }\left( ta+\left(
1-t\right) b\right) \right\vert ^{q}dt  \label{E2} \\
&\leq &\int_{\frac{b-x}{b-a}}^{1}\left[ t^{s}\left\vert f^{\prime
}(a)\right\vert ^{q}+\left( 1-t\right) ^{s}\left\vert f^{\prime
}(b)\right\vert ^{q}\right] dt  \notag \\
&=&\frac{1}{s+1}\left\{ \left[ 1-\left( \frac{b-x}{b-a}\right) ^{s+1}\right]
\left\vert f^{\prime }(a)\right\vert ^{q}+\left( \frac{x-a}{b-a}\right)
^{s+1}\left\vert f^{\prime }(b)\right\vert ^{q}\right\} .  \notag
\end{eqnarray}%
Further, we have 
\begin{equation}
\int_{0}^{\frac{b-x}{b-a}}t^{p}dt=\frac{1}{\left( p+1\right) }\left( \frac{%
b-x}{b-a}\right) ^{p+1}  \label{E3}
\end{equation}%
and%
\begin{equation}
\int_{\frac{b-x}{b-a}}^{1}\left( 1-t\right) ^{p}dt=\frac{1}{\left(
p+1\right) }\left( \frac{x-a}{b-a}\right) ^{p+1}.  \label{E4}
\end{equation}%
A combination of (\ref{E1})-(\ref{E4}) gives the required inequality (\ref%
{E0}).
\end{proof}

\begin{remark}
In Theorem \ref{teo1}, if we choose $x=\frac{a+b}{2}$ and $s=1$, then we have%
\begin{eqnarray*}
&&\left\vert f(\frac{a+b}{2})-\frac{1}{b-a}\int_{a}^{b}f(u)du\right\vert \\
&\leq &\dfrac{\left( b-a\right) }{16}\left( \frac{4}{p+1}\right) ^{\frac{1}{p%
}}\left[ \left( \left\vert f^{\prime }(a)\right\vert ^{q}+3\left\vert
f^{\prime }(b)\right\vert ^{q}\right) ^{1/q}+\left( 3\left\vert f^{\prime
}(a)\right\vert ^{q}+\left\vert f^{\prime }(b)\right\vert ^{q}\right) ^{1/q}%
\right]
\end{eqnarray*}%
which is (\ref{e.1.5}).
\end{remark}

\begin{theorem}
\label{t.2.1} Let $f:I\subset \lbrack 0,\infty )\rightarrow \mathbb{R}$ be a
differentiable mapping on $I^{\circ }$ such that $f^{\prime }\in L\left[ a,b%
\right] $, where $a,b\in I$ with $a<b.$ If $\left\vert f^{\prime
}\right\vert ^{q}$ is $s-$convex on $\left[ a,b\right] $, for some fixed $%
s\in \left( 0,1\right] $ and $p>1$, then the following inequality holds:%
\begin{eqnarray*}
&&\left\vert f(x)-\frac{1}{b-a}\int_{a}^{b}f(u)du\right\vert \\
&\leq &\frac{1}{\left( b-a\right) }\frac{1}{\left( p+1\right) ^{\frac{1}{p}}}
\\
&&\times \left\{ \left( b-x\right) ^{2}\left( \frac{\left\vert f^{\prime
}(x)\right\vert ^{q}+\left\vert f^{\prime }(b)\right\vert ^{q}}{s+1}\right)
^{\frac{1}{q}}+\left( x-a\right) ^{2}\left( \frac{\left\vert f^{\prime
}(a)\right\vert ^{q}+\left\vert f^{\prime }(x)\right\vert ^{q}}{s+1}\right)
^{\frac{1}{q}}\right\}
\end{eqnarray*}%
for each $x\in \left[ a,b\right] $, where $\frac{1}{p}+\frac{1}{q}=1.$
\end{theorem}

\begin{proof}
Suppose that $p>1.$ From Lemma \ref{l.1.1} and using the H\"{o}lder
inequality, we have 
\begin{eqnarray*}
&&\left\vert f(x)-\frac{1}{b-a}\int_{a}^{b}f(u)du\right\vert \\
&\leq &\left( b-a\right) \int_{0}^{\frac{b-x}{b-a}}t\left\vert f^{\prime
}\left( ta+\left( 1-t\right) b\right) \right\vert dt \\
&&+\left( b-a\right) \int_{\frac{b-x}{b-a}}^{1}\left\vert t-1\right\vert
\left\vert f^{\prime }\left( ta+\left( 1-t\right) b\right) \right\vert dt \\
&\leq &\left( b-a\right) \left( \int_{0}^{\frac{b-x}{b-a}}t^{p}dt\right) ^{%
\frac{1}{p}}\left( \int_{0}^{\frac{b-x}{b-a}}\left\vert f^{\prime }\left(
ta+\left( 1-t\right) b\right) \right\vert ^{q}dt\right) ^{\frac{1}{q}} \\
&&+\left( b-a\right) \left( \int_{\frac{b-x}{b-a}}^{1}\left( 1-t\right)
^{p}dt\right) ^{\frac{1}{p}}\left( \int_{\frac{b-x}{b-a}}^{1}\left\vert
f^{\prime }\left( ta+\left( 1-t\right) b\right) \right\vert ^{q}dt\right) ^{%
\frac{1}{q}}.
\end{eqnarray*}%
Since $\left\vert f^{\prime }\right\vert ^{q}$ is $s-$convex, by (\ref{e.1.3}%
) we have%
\begin{equation}
\int_{0}^{\frac{b-x}{b-a}}\left\vert f^{\prime }\left( ta+\left( 1-t\right)
b\right) \right\vert ^{q}dt\leq \frac{b-x}{b-a}\left( \frac{\left\vert
f^{\prime }(x)\right\vert ^{q}+\left\vert f^{\prime }(b)\right\vert ^{q}}{s+1%
}\right)  \label{e.2.1}
\end{equation}%
and%
\begin{equation}
\int_{\frac{b-x}{b-a}}^{1}\left\vert f^{\prime }\left( ta+\left( 1-t\right)
b\right) \right\vert ^{q}dt\leq \frac{x-a}{b-a}\left( \frac{\left\vert
f^{\prime }(a)\right\vert ^{q}+\left\vert f^{\prime }(x)\right\vert ^{q}}{s+1%
}\right) .  \label{e.2.2}
\end{equation}%
Therefore, 
\begin{eqnarray*}
&&\left\vert f(x)-\frac{1}{b-a}\int_{a}^{b}f(u)du\right\vert \\
&\leq &\frac{1}{\left( b-a\right) }\frac{1}{\left( p+1\right) ^{\frac{1}{p}}}
\\
&&\times \left\{ \left( b-x\right) ^{2}\left( \frac{\left\vert f^{\prime
}(x)\right\vert ^{q}+\left\vert f^{\prime }(b)\right\vert ^{q}}{s+1}\right)
^{\frac{1}{q}}+\left( x-a\right) ^{2}\left( \frac{\left\vert f^{\prime
}(a)\right\vert ^{q}+\left\vert f^{\prime }(x)\right\vert ^{q}}{s+1}\right)
^{\frac{1}{q}}\right\}
\end{eqnarray*}%
where $\frac{1}{p}+\frac{1}{q}=1.$ Also, we note that 
\begin{equation*}
\int_{0}^{\frac{b-x}{b-a}}t^{p}dt=\frac{1}{p+1}\left( \frac{b-x}{b-a}\right)
^{p+1}
\end{equation*}%
and%
\begin{equation*}
\int_{\frac{b-x}{b-a}}^{1}\left( 1-t\right) ^{p}dt=\frac{1}{p+1}\left( \frac{%
x-a}{b-a}\right) ^{p+1}.
\end{equation*}%
This completes the proof.
\end{proof}

\begin{remark}
We choose $\left\vert f^{\prime }(x)\right\vert \leq M,$ $M>0$ in Theorem %
\ref{t.2.1}, then we recapture the inequality (\ref{EE}).
\end{remark}

\begin{corollary}
\label{c.2.2.} In Theorem \ref{t.2.1}, if we choose $x=\frac{a+b}{2},$ then%
\begin{eqnarray*}
&&\left\vert f(\frac{a+b}{2})-\frac{1}{b-a}\int_{a}^{b}f(u)du\right\vert \\
&\leq &\frac{b-a}{4}\frac{1}{\left( p+1\right) ^{\frac{1}{p}}}\left\{ \left( 
\frac{\left\vert f^{\prime }(\frac{a+b}{2})\right\vert ^{q}+\left\vert
f^{\prime }(b)\right\vert ^{q}}{s+1}\right) ^{\frac{1}{q}}+\left( \frac{%
\left\vert f^{\prime }(a)\right\vert ^{q}+\left\vert f^{\prime }(\frac{a+b}{2%
})\right\vert ^{q}}{s+1}\right) ^{\frac{1}{q}}\right\} .
\end{eqnarray*}
\end{corollary}

\begin{corollary}
\label{cor2} In Corollary \ref{c.2.2.}, if we choose $f^{\prime
}(a)=f^{\prime }\left( \frac{a+b}{2}\right) =f^{\prime }(b)$ and $s=1$, then%
\begin{equation}
\left\vert f(\frac{a+b}{2})-\frac{1}{b-a}\int_{a}^{b}f(u)du\right\vert \leq 
\frac{b-a}{\left( p+1\right) ^{\frac{1}{p}}}\left( \frac{\left\vert
f^{\prime }(b)\right\vert +\left\vert f^{\prime }(a)\right\vert }{4}\right) .
\label{E5}
\end{equation}
\end{corollary}

\begin{remark}
We note that the obtained midpoint inequality (\ref{E5}) is better than the
inequality (\ref{e.1.6}).
\end{remark}

\begin{theorem}
\label{z} Let $f:I\subset \lbrack 0,\infty )\rightarrow \mathbb{R}$ be a
differentiable mapping on $I^{\circ }$ such that $f^{\prime }\in L\left[ a,b%
\right] $, where $a,b\in I$ with $a<b.$ If $\left\vert f^{\prime
}\right\vert ^{q}$ is $s-$convex on $\left[ a,b\right] $, for some fixed $%
s\in \left( 0,1\right] $ and $p>1$, then the following inequality holds:%
\begin{eqnarray}
&&\left\vert f(x)-\frac{1}{b-a}\int_{a}^{b}f(u)du\right\vert  \notag \\
&&  \label{z0} \\
&\leq &\frac{\left( b-a\right) }{\left( p+1\right) ^{\frac{1}{p}}}\left[
\left( \frac{b-x}{b-a}\right) ^{p+1}+\left( \frac{x-a}{b-a}\right) ^{p+1}%
\right] ^{\frac{1}{p}}\left( \frac{\left\vert f^{\prime }(a)\right\vert
^{q}+\left\vert f^{\prime }(b)\right\vert ^{q}}{s+1}\right) ^{\frac{1}{q}} 
\notag
\end{eqnarray}%
for each $x\in \left[ a,b\right] $, where $\frac{1}{p}+\frac{1}{q}=1.$
\end{theorem}

\begin{proof}
Suppose that $p>1.$ From Lemma \ref{l.1.1} and using the H\"{o}lder
inequality, we have 
\begin{eqnarray}
&&\left\vert f(x)-\frac{1}{b-a}\int_{a}^{b}f(u)du\right\vert  \label{z2} \\
&&  \notag \\
&\leq &\left( b-a\right) \left( \int_{0}^{1}\left\vert p(t)\right\vert
^{p}dt\right) ^{\frac{1}{p}}\left( \int_{0}^{1}\left\vert f^{\prime }\left(
ta+\left( 1-t\right) b\right) \right\vert ^{q}dt\right) ^{\frac{1}{q}}. 
\notag
\end{eqnarray}%
Since $\left\vert f^{\prime }\right\vert ^{q}$ is $s-$convex, we have%
\begin{eqnarray}
\int_{0}^{1}\left\vert f^{\prime }\left( ta+\left( 1-t\right) b\right)
\right\vert ^{q}dt &\leq &\int_{0}^{1}\left( t^{s}\left\vert f^{\prime
}\left( a\right) \right\vert ^{q}+\left( 1-t\right) ^{s}\left\vert f^{\prime
}\left( b\right) \right\vert ^{q}\right) dt  \notag \\
&&  \label{z3} \\
&=&\frac{\left\vert f^{\prime }\left( a\right) \right\vert ^{q}+\left\vert
f^{\prime }\left( b\right) \right\vert ^{q}}{s+1}  \notag
\end{eqnarray}%
and%
\begin{eqnarray}
\int_{0}^{1}\left\vert p(t)\right\vert ^{p}dt &=&\dint_{0}^{\frac{b-x}{b-a}%
}t^{p}dt+\dint_{\frac{b-x}{b-a}}^{1}\left( 1-t\right) ^{p}dt  \notag \\
&&  \label{z4} \\
&=&\frac{1}{p+1}\left[ \left( \frac{b-x}{b-a}\right) ^{p+1}+\left( \frac{x-a%
}{b-a}\right) ^{p+1}\right] .  \notag
\end{eqnarray}%
Using (\ref{z3}) and (\ref{z4}) in (\ref{z2}), we obtain (\ref{z0}).
\end{proof}

\begin{corollary}
\label{zz} Under assumptation in Theorem \ref{z} with $p=q=2,$ we have 
\begin{equation*}
\left\vert f(x)-\frac{1}{b-a}\int_{a}^{b}f(u)du\right\vert \leq \frac{\left(
b-a\right) }{\sqrt{3}}\left[ \frac{1}{4}+\frac{\left( x-\frac{a+b}{2}\right)
^{2}}{(b-a)^{2}}\right] ^{\frac{1}{2}}\left( \frac{\left\vert f^{\prime
}(a)\right\vert ^{2}+\left\vert f^{\prime }(b)\right\vert ^{2}}{s+1}\right)
^{\frac{1}{2}}.
\end{equation*}
\end{corollary}

\begin{corollary}
\label{zz1} In Corollary \ref{zz}, if we choose $x=\frac{a+b}{2}$ and $s=1,$
then we have the following midpoint inequality:%
\begin{equation*}
\left\vert f(\frac{a+b}{2})-\frac{1}{b-a}\int_{a}^{b}f(u)du\right\vert \leq 
\frac{\left( b-a\right) }{2}\left( \frac{\left\vert f^{\prime
}(a)\right\vert ^{2}+\left\vert f^{\prime }(b)\right\vert ^{2}}{6}\right) ^{%
\frac{1}{2}}.
\end{equation*}
\end{corollary}

\begin{remark}
We choose $\left\vert f^{\prime }(x)\right\vert \leq M,$ $M>0$ and $s=1$ in
Corollary \ref{zz}, then we recapture the following Ostrowski's type
inequality%
\begin{equation*}
\left\vert f(x)-\frac{1}{b-a}\int_{a}^{b}f(u)du\right\vert \leq \frac{%
M\left( b-a\right) }{\sqrt{3}}\left[ \frac{1}{4}+\frac{\left( x-\frac{a+b}{2}%
\right) ^{2}}{(b-a)^{2}}\right] ^{\frac{1}{2}}.
\end{equation*}
\end{remark}

\begin{theorem}
\label{t.2.2} Let $f:I\subset \lbrack 0,\infty )\rightarrow \mathbb{R}$ be a
differentiable mapping on $I^{\circ }$ such that $f^{\prime }\in L\left[ a,b%
\right] $, where $a,b\in I$ with $a<b.$ If $\left\vert f^{\prime
}\right\vert ^{q}$ is $s-$convex on $\left[ a,b\right] $, for some fixed $%
s\in \left( 0,1\right] $ and $q\geq 1$, then the following inequality holds:%
\begin{eqnarray*}
&&\left\vert f(x)-\frac{1}{b-a}\int_{a}^{b}f(u)du\right\vert \\
&\leq &\left( b-a\right) \left( \frac{1}{2}\right) ^{1-\frac{1}{q}}\left\{
\left( \frac{b-x}{b-a}\right) ^{2(1-1/q)}\left[ \frac{1}{s+2}\left( \frac{b-x%
}{b-a}\right) ^{s+2}\left\vert f^{\prime }(a)\right\vert ^{q}\right. \right.
\\
&&\left. \left. +\left( \frac{1}{s+2}\left( \frac{x-a}{b-a}\right) ^{s+2}-%
\frac{1}{s+1}\left( \frac{x-a}{b-a}\right) ^{s+1}+\frac{1}{\left( s+1\right)
\left( s+2\right) }\right) \left\vert f^{\prime }(b)\right\vert ^{q}\right]
^{\frac{1}{q}}\right\} \\
&&+\left( b-a\right) \left( \frac{1}{2}\right) ^{1-\frac{1}{q}}\left\{
\left( \frac{x-a}{b-a}\right) ^{2(1-1/q)}\left[ \frac{1}{s+2}\left( \frac{x-a%
}{b-a}\right) ^{s+2}\left\vert f^{\prime }(b)\right\vert ^{q}\right. \right.
\\
&&\left. \left. +\left( \frac{1}{s+2}\left( \frac{b-x}{b-a}\right) ^{s+2}-%
\frac{1}{s+1}\left( \frac{b-x}{b-a}\right) ^{s+1}+\frac{1}{\left( s+1\right)
\left( s+2\right) }\right) \left\vert f^{\prime }(a)\right\vert ^{q}\right]
^{\frac{1}{q}}\right\}
\end{eqnarray*}%
for each $x\in \left[ a,b\right] .$
\end{theorem}

\begin{proof}
Suppose that $q\geq 1.$ From Lemma \ref{l.1.1} and using the well known
power mean inequality, we have 
\begin{eqnarray*}
&&\left\vert f(x)-\frac{1}{b-a}\int_{a}^{b}f(u)du\right\vert \\
&\leq &\left( b-a\right) \dint_{0}^{\frac{b-x}{b-a}}t\left\vert f^{\prime
}\left( ta+\left( 1-t\right) b\right) \right\vert dt \\
&&+\left( b-a\right) \dint_{\frac{b-x}{b-a}}^{1}\left\vert t-1\right\vert
\left\vert f^{\prime }\left( ta+\left( 1-t\right) b\right) \right\vert dt \\
&\leq &\left( b-a\right) \left( \dint_{0}^{\frac{b-x}{b-a}}tdt\right) ^{1-%
\frac{1}{q}}\left( \dint_{0}^{\frac{b-x}{b-a}}t\left\vert f^{\prime }\left(
ta+\left( 1-t\right) b\right) \right\vert ^{q}dt\right) ^{\frac{1}{q}} \\
&&+\left( b-a\right) \left( \dint_{\frac{b-x}{b-a}}^{1}\left( 1-t\right)
dt\right) ^{1-\frac{1}{q}}\left( \dint_{\frac{b-x}{b-a}}^{1}\left(
1-t\right) \left\vert f^{\prime }\left( ta+\left( 1-t\right) b\right)
\right\vert ^{q}dt\right) ^{\frac{1}{q}}.
\end{eqnarray*}%
Since $\left\vert f^{\prime }\right\vert ^{q}$ is $s-$convex, we have%
\begin{eqnarray*}
&&\int_{0}^{\frac{b-x}{b-a}}t\left\vert f^{\prime }\left( ta+\left(
1-t\right) b\right) \right\vert ^{q}dt \\
&\leq &\int_{0}^{\frac{b-x}{b-a}}t\left[ t^{s}\left\vert f^{\prime
}(a)\right\vert ^{q}+\left( 1-t\right) ^{s}\left\vert f^{\prime
}(b)\right\vert ^{q}\right] dt \\
&=&\frac{1}{s+2}\left( \frac{b-x}{b-a}\right) ^{s+2}\left\vert f^{\prime
}(a)\right\vert ^{q}+\left[ \frac{1}{s+2}\left( \frac{x-a}{b-a}\right)
^{s+2}\right. \\
&&\left. -\frac{1}{\left( s+1\right) }\left( \frac{x-a}{b-a}\right) ^{s+1}+%
\frac{1}{\left( s+1\right) \left( s+2\right) }\right] \left\vert f^{\prime
}(b)\right\vert ^{q}
\end{eqnarray*}%
and%
\begin{eqnarray*}
&&\int_{\frac{b-x}{b-a}}^{1}\left( 1-t\right) \left\vert f^{\prime }\left(
ta+\left( 1-t\right) b\right) \right\vert ^{q}dt \\
&\leq &\int_{\frac{b-x}{b-a}}^{1}\left( 1-t\right) \left[ t^{s}\left\vert
f^{\prime }(a)\right\vert ^{q}+\left( 1-t\right) ^{s}\left\vert f^{\prime
}(b)\right\vert ^{q}\right] dt \\
&=&\left( \frac{1}{\left( s+1\right) \left( s+2\right) }-\frac{1}{s+1}\left( 
\frac{b-x}{b-a}\right) ^{s+1}+\frac{1}{s+2}\left( \frac{b-x}{b-a}\right)
^{s+2}\right) \left\vert f^{\prime }(a)\right\vert ^{q} \\
&&+\frac{1}{s+2}\left( \frac{x-a}{b-a}\right) ^{s+2}\left\vert f^{\prime
}(b)\right\vert ^{q}.
\end{eqnarray*}%
Therefore we have%
\begin{eqnarray*}
&&\left\vert f(x)-\frac{1}{b-a}\int_{a}^{b}f(u)du\right\vert \\
&\leq &\left( b-a\right) \left( \frac{1}{2}\right) ^{1-\frac{1}{q}}\left\{
\left( \frac{b-x}{b-a}\right) ^{2(1-1/q)}\left[ \frac{1}{s+2}\left( \frac{b-x%
}{b-a}\right) ^{s+2}\left\vert f^{\prime }(a)\right\vert ^{q}\right. \right.
\\
&&\left. \left. +\left( \frac{1}{s+2}\left( \frac{x-a}{b-a}\right) ^{s+2}-%
\frac{1}{s+1}\left( \frac{x-a}{b-a}\right) ^{s+1}+\frac{1}{\left( s+1\right)
\left( s+2\right) }\right) \left\vert f^{\prime }(b)\right\vert ^{q}\right]
^{\frac{1}{q}}\right\} \\
&&+\left( b-a\right) \left( \frac{1}{2}\right) ^{1-\frac{1}{q}}\left\{
\left( \frac{x-a}{b-a}\right) ^{2(1-1/q)}\left[ \frac{1}{s+2}\left( \frac{x-a%
}{b-a}\right) ^{s+2}\left\vert f^{\prime }(b)\right\vert ^{q}\right. \right.
\\
&&\left. \left. +\left( \frac{1}{s+2}\left( \frac{b-x}{b-a}\right) ^{s+2}-%
\frac{1}{s+1}\left( \frac{b-x}{b-a}\right) ^{s+1}+\frac{1}{\left( s+1\right)
\left( s+2\right) }\right) \left\vert f^{\prime }(a)\right\vert ^{q}\right]
^{\frac{1}{q}}\right\}
\end{eqnarray*}%
which is required.
\end{proof}

\begin{corollary}
\label{c.2.3.} In Theorem \ref{t.2.2}, if we choose $x=\frac{a+b}{2}$ and $%
s=1,$ then we have%
\begin{eqnarray*}
&&\left\vert f(\frac{a+b}{2})-\frac{1}{b-a}\int_{a}^{b}f(u)du\right\vert \\
&\leq &\frac{b-a}{8}\left( \frac{1}{3}\right) ^{\frac{1}{q}}\left[ \left(
\left\vert f^{\prime }(a)\right\vert ^{q}+3\left\vert f^{\prime
}(b)\right\vert ^{q}\right) ^{\frac{1}{q}}+\left( 3\left\vert f^{\prime
}(a)\right\vert ^{q}+\left\vert f^{\prime }(b)\right\vert ^{q}\right) ^{%
\frac{1}{q}}\right] .
\end{eqnarray*}
\end{corollary}

\section{Applications To Special Means}

Let $0<s<1$ and $u,v,w\in \mathbb{R}$. We define a function $f:\left[
0,\infty \right) \rightarrow \mathbb{R}$%
\begin{equation*}
f(t)=\left\{ 
\begin{array}{lll}
u & if & t=0 \\ 
&  &  \\ 
vt^{s}+w & if & t>0.%
\end{array}%
\right.
\end{equation*}%
If $v\geq 0$ and $0\leq w\leq u$, then $f\in K_{s}^{2}$ (see \cite{hudzik}).
Hence, for $u=w=0$, $v=1$, we have $f:\left[ 0,1\right] \rightarrow \left[
0,1\right] $, $f(t)=t^{s},$ $f\in K_{s}^{2}.$

As in \cite{SSO}, we shall consider the means for arbitrary positive real
numbers $a,b,$ $a\neq b.$ We take

(1) The arithmetic mean:%
\begin{equation*}
A=A(a,b):=\dfrac{a+b}{2},
\end{equation*}

(2) The logarithmic mean: 
\begin{equation*}
L=L\left( a,b\right) :=\left\{ 
\begin{array}{ccc}
a & if & a=b \\ 
&  &  \\ 
\frac{b-a}{\ln b-\ln a} & if & a\neq b%
\end{array}%
\right. \text{,}
\end{equation*}

(3) The $p-$logarithmic mean:

\begin{equation*}
L_{p}=L_{p}(a,b):=\left\{ 
\begin{array}{ccc}
\left[ \frac{b^{p+1}-a^{p+1}}{\left( p+1\right) \left( b-a\right) }\right] ^{%
\frac{1}{p}} & \text{if} & a\neq b \\ 
&  &  \\ 
a & \text{if} & a=b%
\end{array}%
\right. \text{, \ \ \ }p\in \mathbb{R\diagdown }\left\{ -1,0\right\} .
\end{equation*}%
It is well known \ that $L_{p}$ is monotonic nondecreasing \ over $p\in 
\mathbb{R}$ with $L_{-1}:=L$ and $L_{0}:=I.$ In particular, we have the
following inequality%
\begin{equation*}
L\leq A.
\end{equation*}%
Now, using the results of Section 2, we give some applications to special
means of positive real numbers.

\begin{proposition}
\label{p1} Let $0<a<b$ and $s\in \left( 0,1\right) .$ Then we have 
\begin{equation*}
\left\vert A^{s}\left( a,b\right) -L_{s}^{s}\left( a,b\right) \right\vert
\leq \left( b-a\right) \frac{s}{\left( s+1\right) \left( s+2\right) }\left(
1-\frac{1}{2^{s+1}}\right) \left[ a^{s-1}+b^{s-1}\right] .
\end{equation*}
\end{proposition}

\begin{proof}
The inequality follows from (\ref{e.2.0}) applied to the $s-$convex function
in the second sense $f:\left[ 0,1\right] \rightarrow \left[ 0,1\right] ,$ $%
f(x)=x^{s}.$ The details are omitted.
\end{proof}

\begin{proposition}
\label{p2} Let $0<a<b$ and $s\in \left( 0,1\right) .$ Then we have 
\begin{eqnarray*}
&&\left\vert A^{s}\left( a,b\right) -L_{s}^{s}\left( a,b\right) \right\vert
\\
&\leq &s\frac{\left( b-a\right) }{4}\frac{1}{\left( p+1\right) ^{1/p}}\left[
\left( \frac{A^{q(s-1)}+b^{q(s-1)}}{s+1}\right) ^{1/q}+\left( \frac{%
a^{q(s-1)}+A^{q(s-1)}}{s+1}\right) ^{1/q}\right] .
\end{eqnarray*}
\end{proposition}

\begin{proof}
The proof is similar to that of Proposition \ref{p1}, using Corollary \ref%
{c.2.2.}.
\end{proof}

\begin{proposition}
\label{p3} Let $0<a<b$ and $s\in \left( 0,1\right) .$ Then we have 
\begin{eqnarray*}
&&\left\vert A^{s}\left( a,b\right) -L_{s}^{s}\left( a,b\right) \right\vert
\\
&\leq &s\frac{\left( b-a\right) }{8}\left( \frac{2}{3}\right) ^{\frac{1}{q}%
}\left\{ \left[ A\left( a^{q(s-1)},3b^{q(s-1)}\right) \right] ^{1/q}+\left[
A\left( 3a^{q(s-1)},b^{q(s-1)}\right) \right] ^{1/q}\right\} .
\end{eqnarray*}
\end{proposition}

\begin{proof}
The proof is similar to that of Proposition \ref{p1}, using Corollary \ref%
{c.2.3.}.
\end{proof}

\section{The Midpoint Formula}

As in \cite{USKMEO} and \cite{PP}, let $d$ be a division $%
a=x_{0}<x_{1}<...<x_{n-1}<x_{n}=b$ of the interval $\left[ a,b\right] $ and
consider the quadrature formula 
\begin{equation}
\int_{a}^{b}f(x)dx=T\left( f,d\right) +E(f,d)  \label{e.4.1}
\end{equation}%
where%
\begin{equation*}
T\left( f,d\right) =\sum_{i=0}^{n-1}f\left( \frac{x_{i}+x_{i+1}}{2}\right)
\left( x_{i+1}-x_{i}\right)
\end{equation*}%
is the midpoint version and $E(f,d)$ denotes the associated approximation
error.

In the following, we propose some new estimates for midpoint formula.

\begin{proposition}
\label{p4} Let $f:I\subset \lbrack 0,\infty )\rightarrow \mathbb{R}$ be a
differentiable mapping on $I^{\circ }$ such that $f^{\prime }\in L\left[ a,b%
\right] $, where $a,b\in I$ with $a<b.$ If $\left\vert f^{\prime
}\right\vert ^{q}$ is convex on $\left[ a,b\right] $ and $p>1$, then in (\ref%
{e.4.1}), for every division $d$ of $\left[ a,b\right] $, the midpoint error
satisfy 
\begin{equation*}
\left\vert E(f,d)\right\vert \leq \frac{1}{4\left( p+1\right) ^{\frac{1}{p}}}%
\sum_{i=0}^{n-1}\left( x_{i+1}-x_{i}\right) ^{2}\left[ \left\vert f^{\prime
}(x_{i+1})\right\vert +\left\vert f^{\prime }(x_{i})\right\vert \right] .
\end{equation*}
\end{proposition}

\begin{proof}
On applying Corollary \ref{cor2} on the subinterval $\left[ x_{i},x_{i+1}%
\right] \left( i=0,1,...,n-1\right) $ of the division , we get 
\begin{equation*}
\left\vert \left( x_{i+1}-x_{i}\right) f\left( \frac{x_{i}+x_{i+1}}{2}%
\right) -\int_{x_{i}}^{x_{i+1}}f(x)dx\right\vert \leq \frac{\left(
x_{i+1}-x_{i}\right) ^{2}}{(p+1)^{1/p}}\left[ \frac{\left\vert f^{\prime
}(x_{i+1})\right\vert +\left\vert f^{\prime }(x_{i})\right\vert }{4}\right] .
\end{equation*}%
Summing over $i$ from $0$ to $n-1$ and taking into account that $\left\vert
f^{\prime }\right\vert $ is convex, we obtain, by the triangle inequality,
that 
\begin{eqnarray*}
&&\left\vert \int_{a}^{b}f(x)dx-T\left( f,d\right) \right\vert \\
&=&\left\vert \sum_{i=0}^{n-1}\left\{ \int_{x_{i}}^{x_{i+1}}f(x)dx-f\left( 
\frac{x_{i}+x_{i+1}}{2}\right) \left( x_{i+1}-x_{i}\right) \right\}
\right\vert \\
&\leq &\sum_{i=0}^{n-1}\left\vert \left\{
\int_{x_{i}}^{x_{i+1}}f(x)dx-f\left( \frac{x_{i}+x_{i+1}}{2}\right) \left(
x_{i+1}-x_{i}\right) \right\} \right\vert \\
&\leq &\frac{1}{4(p+1)^{1/p}}\sum_{i=0}^{n-1}\left( x_{i+1}-x_{i}\right) ^{2}%
\left[ \left\vert f^{\prime }(x_{i+1})\right\vert +\left\vert f^{\prime
}(x_{i})\right\vert \right]
\end{eqnarray*}%
which is completed the proof.
\end{proof}

\begin{proposition}
\label{p5} Let $f:I\subset \lbrack 0,\infty )\rightarrow \mathbb{R}$ be a
differentiable mapping on $I^{\circ }$ such that $f^{\prime }\in L\left[ a,b%
\right] $, where $a,b\in I$ with $a<b.$ If $\left\vert f^{\prime
}\right\vert ^{q}$ is convex on $\left[ a,b\right] $, then in (\ref{e.4.1}),
for every division $d$ of $\left[ a,b\right] $, the midpoint error satisfy 
\begin{equation*}
\left\vert E(f,d)\right\vert \leq \frac{1}{2\sqrt{6}}\sum_{i=0}^{n-1}\left(
x_{i+1}-x_{i}\right) ^{2}\left[ \left\vert f^{\prime }(x_{i+1})\right\vert
^{2}+\left\vert f^{\prime }(x_{i})\right\vert ^{2}\right] ^{1/2}.
\end{equation*}
\end{proposition}

\begin{proof}
The proof uses Corollary \ref{zz1} and is similar to that of Proposition \ref%
{p4}.
\end{proof}

\begin{proposition}
\label{p6} Let $f:I\subset \lbrack 0,\infty )\rightarrow \mathbb{R}$ be a
differentiable mapping on $I^{\circ }$ such that $f^{\prime }\in L\left[ a,b%
\right] $, where $a,b\in I$ with $a<b.$ If $\left\vert f^{\prime
}\right\vert ^{q}$ is convex on $\left[ a,b\right] $ and $q\geq 1$, then in (%
\ref{e.4.1}), for every division $d$ of $\left[ a,b\right] $, the midpoint
error satisfy 
\begin{multline*}
\left\vert E(f,d)\right\vert \\
\leq \frac{1}{8}\left( \frac{1}{3}\right) ^{\frac{1}{q}}\sum_{i=0}^{n-1}%
\left( x_{i+1}-x_{i}\right) ^{2}\left[ \left( \left\vert f^{\prime
}(x_{i})\right\vert ^{q}+3\left\vert f^{\prime }(x_{i+1})\right\vert
^{q}\right) ^{\frac{1}{q}}+\left( 3\left\vert f^{\prime }(x_{i})\right\vert
^{q}+\left\vert f^{\prime }(x_{i+1})\right\vert ^{q}\right) ^{\frac{1}{q}}%
\right] .
\end{multline*}
\end{proposition}

\begin{proof}
The proof uses Corollary \ref{c.2.3.} and is similar to that of Proposition %
\ref{p4}.
\end{proof}

\end{document}